\documentclass[12pt,reqno]{amsart}
\usepackage{amsmath,amsfonts,amsthm,amssymb,amsxtra}

\newcommand{\bbN}{{\mathbb{N}}}

\newcommand{\bbR}{{\mathbb{R}}}

\newcommand{\bbC}{{\mathbb{C}}}

\newcommand{\no}{\nonumber}

\newcommand{\ti}{\tilde  }

\newcommand{\dom}{\text{\rm{dom}}}

\newcommand{\supp}{\text{\rm{supp}}}

\newcommand{\beq}{\begin{equation}}

\newcommand{\eeq}{\end{equation}}

\newcommand{\ba}{\begin{align}}

\newcommand{\ea}{\end{align}}

\allowdisplaybreaks

\numberwithin{equation}{section}

\newcommand{\C}{\mathbb{C}}

\renewcommand{\epsilon}{\varepsilon}

\newcommand{\N}{\mathbb{N}}

\renewcommand{\phi}{\varphi}

\newcommand{\R}{\mathbb{R}}

\DeclareMathOperator{\const}{const}

\DeclareMathOperator{\im}{Im}

\DeclareMathOperator{\sgn}{sgn}

\newtheorem{theorem}{Theorem}

\newtheorem{proposition}[theorem]{Proposition}

\newtheorem{lemma}[theorem]{Lemma}

\newtheorem{corollary}[theorem]{Corollary}

\theoremstyle{definition}

\newtheorem{definition}[theorem]{Definition}

\theoremstyle{remark}

\newtheorem*{remark*}{Remark}

\begin{document}

\title{Singular spectrum for radial trees}

\author{Jonathan Breuer}
\address{Jonathan Breuer, Mathematics 253-37, California Institute of Technology, Pasadena, CA 91125, USA}
\email{jbreuer@caltech.edu}

\author{Rupert L. Frank}
\address{Rupert L. Frank, Department of Mathematics,
Princeton University, Washington Road, Princeton, NJ 08544, USA}
\email{rlfrank@math.princeton.edu}

\begin{abstract}
We prove several results showing that absolutely continuous spectrum for the Laplacian on 
radial trees is a rare event. In particular, we show that metric trees with unbounded edges 
have purely singular spectrum and that generically (in the sense of Baire) radial trees have purely 
singular continuous spectrum.
\end{abstract}

\date{}

\maketitle


\sloppy

\section{Introduction and main results}

This paper deals with the absence of absolutely continuous spectrum for the Laplacian on radial trees. 
In particular, our purpose is to demonstrate that the existence of absolutely continuous spectrum imposes 
rather stringent restrictions on the structure of the tree, so that generally, the occurrence of absolutely 
continuous spectrum is a rather exceptional event. 
While our primary concern are metric trees, let us demonstrate the above first in the discrete setting.
   
Let $\Gamma_d$ be a radial discrete tree, that is, a rooted,
radially symmetric tree graph. Denote the root by $O$ and, for any
vertex $x$, the branching number of $x$ (that is, the number of forward
nearest neighbors of $x$) by $b(x)$. The symmetry implies that
$b(x)$ is a function of the distance of $x$ to the root. We only consider infinite trees. Therefore,
there exists a sequence of natural numbers, $\{b_n\}_{n=1}^\infty$,
such that if $n=d(O,x)$ then $b(x)=b_n$.

We have

\begin{theorem} \label{discrete}
Assume $\{b_n\}_{n=1}^\infty$ is bounded. Let $\Delta$ be the
discrete Laplacian on $\Gamma_d$. Then, if $\Delta$ has 
nonempty absolutely continuous spectrum, the sequence
$\{b_n\}_{n=1}^\infty$ is eventually periodic. That is, there exists
$N \in \bbN$ such that $\{b_n\}_{n=N}^\infty$ is a (one sided)
periodic sequence.
\end{theorem}
\begin{remark*}
Different authors use different definitions for $\Delta$. The
theorem holds both for

\beq \no (\Delta f)(x)=\sum_{y \sim x}f(y) \eeq (AKA the adjacency
matrix), and for
\beq \no (\Delta f)(x)= \left\{
\begin{array}{cc} \left(b(x)+1 \right)f(x)-\sum_{y \sim x}f(y) & x \neq O \\
\left(b(x)\right)f(x)- \sum_{y \sim x}f(y) & x=O \end{array} \right.
\eeq where $y \sim x$ means $y$ is a nearest neighbor of $x$. 
\end{remark*}

Regarding radial \emph{metric} trees (see, e.g., \cite{solomyak}
for an introduction to the concept of metric trees), a moment's
reflection shows that such a result can not hold since the edge lengths 
are now continuous parameters. We can, however, show that unbounded edge lengths does rule out absolutely 
continuous spectrum.

Explicitly, let $\Gamma_c$ be a radial metric tree with root $O$. As
before, let $b(x)$ be the branching number of a vertex $x$. We assume $b(O)=1$
and $b(x)>1$ for any other vertex. If $x$ is a vertex with $n+1$
vertices on the unique geodesic connecting it to $O$ (including the
endpoints), we denote $d(O,x)=t_n$ and $b(x)=b_n$. The parameters $b_n$ and $t_n$
are well defined because of the radial symmetry. We shall assume
that
\begin{enumerate}
\item $\inf_n \left(t_{n+1}-t_n\right)>0$.
\end{enumerate}

\begin{remark*}
In the metric tree literature, the trees we consider here are
usually called \emph{regular trees} \cite{naim-solom,solomyak}. Since this paper also 
has a theorem about discrete trees, where regularity usually means that every vertex has
the same number of neighbors, we chose the term \emph{radial} for
both settings as a unifying compromise.
\end{remark*}

Let $-\Delta$ be the operator in $L^2(\Gamma_c)$
defined through the quadratic form \beq \no \int_{\Gamma_c} |u'(x)|^2\,dx
\eeq for functions $u \in H^1_0(\Gamma_c)$, the Sobolev space of
functions continuous along the edges and satisfying $\int_{\Gamma_c}
\left(|u'(x)|^2+|u(x)|^2 \right)\,dx  <\infty$, $u(O)=0$. Functions
in the operator domain of $-\Delta$ satisfy Dirichlet boundary conditions at
the root and Kirchhoff boundary conditions at the vertices.

We shall prove

\begin{theorem}\label{continuous}
Under assumption (1) above, if 
\beq \no
\limsup_{n \rightarrow \infty}(t_{n+1}-t_n)=\infty \,,
\eeq 
then the spectrum of $-\Delta$ is purely singular, in the 
sense that any spectral measure for $-\Delta$ is supported on a set of Lebesgue measure zero.
\end{theorem}

Put differently, this theorem says that \emph{any} subsequence of unbounded edges destroys absolutely continuous spectrum, 
no matter what happens between these unbounded edges. 
The trees described in Theorem \ref{continuous} are sparse in the sense that their branchings become sparse as the 
distance from the root increases (at least along some sequence). The fact that sparse potentials for one-dimensional
Schr\"odinger operators may lead to singular spectral measures was discovered by Pearson \cite{pearson} thirty years 
ago and has been explored extensively since then (see \cite{last-review} and references therein). The paper  
\cite{remling-disc}, whose main result is one of the principal ingredients in our analysis, has a remarkably general 
result (Corollary 1.5 there) in this area. 

A restricted family of discrete sparse trees was 
studied in \cite{breuer-sparse, breuer-mol}, where it was shown that by controlling the various parameters defining 
the tree, it is possible to control also the `degree of continuity' of the spectral measures. Theorem \ref{continuous} 
above is much softer, but is considerably more general. It is also, to the best of our knowledge, the first theorem 
of this type to be proven for metric trees. We find it remarkable that such a small class of radial metric trees
has absolutely continuous spectrum. The next two theorems show that pure point spectrum does not occur too often as well.

\begin{theorem}\label{simonstolz}
Assume that the sequence $\{b_n\}$ is bounded. Then, if 
\beq \label{sparseness}
\limsup_{n \rightarrow \infty}\frac{t_{n+1}-t_n}{n^{2n}}>0 \,,
\eeq
the spectrum of $-\Delta$ coincides with $[0,\infty)$ and is purely singular continuous.
\end{theorem}

For any $\varepsilon>0$, $C>0$ we consider the set $\mathcal{T}^{\varepsilon,C}$ of radial trees whose 
defining parameter sequences, $\{t_n,b_n\}$, satisfy
\begin{enumerate}
\item[$(1)_\epsilon$] $\inf \left(t_{n+1}-t_n\right)\geq\varepsilon,\qquad t_1 \geq \varepsilon$.
\item[$(2)_C$] $\sup b_n \leq C$.
\end{enumerate}
Moreover, we allow $\{t_n,b_n\}$ to be a finite (possibly empty) sequence. This means that we also consider trees which have only a finite number of vertices and which contain half-lines.

Identifying sequences $\{t_n,b_n\}$ with measures $\sum_n \beta_n \delta_{t_n}$ 
(with $\beta_n=\frac{\sqrt{b_n}+1}{\sqrt{b_n}-1}$)
we can consider 
$\mathcal{T}^{\varepsilon,C}$ as a (compact) metric space with convergence being induced from weak convergence of measures.  
This convergence is natural for us since, as we shall see, 
convergence of the trees implies strong resolvent convergence 
of the corresponding Laplacians (up to a natural unitary transformation taking the 
different Hilbert spaces into account).

{F}rom Theorem \ref{simonstolz} we shall deduce the somewhat surprising

\begin{theorem}\label{wonderland}
In the space $\mathcal{T}^{\varepsilon,C}$, with the topology of weak convergence for the corresponding measures,
the set of trees whose spectrum on $[0,\infty)$ is purely singular continuous, is a dense $G_\delta$ set.
\end{theorem}

Recall that a $G_\delta$ set is a countable intersection of open sets.

As with other works dealing with radial trees (see e.g.\ \cite{breuer-sparse, breuer-mol, carlson, naim-solom, solomyak}), 
the radial symmetry reduces the analysis to a one-dimensional 
problem. Thus, the exclusion of eigenvalues in Theorem \ref{simonstolz} follows by a Simon-Stolz type argument 
\cite{simon-stolz} and Theorem~\ref{wonderland} follows from Theorem \ref{simonstolz} with the help of Simon's 
Wonderland Theorem~\cite{simon-wonderland} applied to the corresponding families of one-dimensional operators. 

The new ingredient which enters in the proofs of Theorems 
\ref{discrete} and \ref{continuous} is a recent theorem of Remling's \cite{remling-disc, remling-cont} 
following work by Breimesser and Pearson \cite{bp1, bp2}. Remling's Theorem leads to various 
explicit restrictions on one-dimensional discrete and continuous Schr\"odinger 
operators with absolutely continuous spectrum. The structural restrictions on trees with 
absolutely continuous spectrum are a consequence of these restrictions. 
In particular, Theorem \ref{discrete} is an immediate corollary of Theorem 1.1 of \cite{remling-disc}, 
given Theorem 2.4 of \cite{breuer-sparse}.

In contrast with the discrete case, Theorem \ref{continuous} is not
an immediate consequence of the results of \cite{remling-disc} or
its continuous counterpart \cite{remling-cont}. The difficulty lies
in the fact that the objects appearing in the direct sum
decomposition of $-\Delta$ (see, e.g., \cite{carlson, naim-solom})
are not `standard' Schr\"odinger operators, but rather
Sturm-Liouville operators on weighted $L^2$ spaces with rather singular weights.

The better part of the rest of this paper is devoted to demonstrating the 
applicability of Remling's Theorem to these operators. A crucial point in the 
analysis is the proof that for a whole-line potential that is reflectionless
(see Section 4.1 for the definition) on a set of positive Lebesgue
measure, the part of the potential lying to the left of $0$ uniquely
determines the part lying to the right of $0$. In the context of
Jacobi matrices and Schr\"odinger operators with measure valued
potentials this is, indeed, a simple realization relying on
classical results. In our case, however, this seems to be a new
result (in particular, see Proposition \ref{unique}). In order to prove this, we have made use of the
Kre\u\i n formula for the difference of the resolvents of two
different self-adjoint extensions of a closed, densely defined,
symmetric operator, as it appears in \cite{posilicano} (see Section 3).
We are not aware of any previous application of this formula in
the spectral theory of Schr\"odinger operators on radial trees, and we believe that it may be useful also beyond the 
context of the present paper. 

Examples of trees, and more general graphs, for which the graph Laplacian has singular spectrum, 
have been constructed before \cite{breuer-sparse, breuer-mol, hislop-post, teplyaev, graph-lap, teplyaev1}. 
However, to the best of our knowledge, the theorems above are the first of their kind in terms of the generality 
in which they hold.
 In particular, we do not know of another Wonderland-type theorem for trees. Interestingly enough, it is not clear 
how to formulate an interesting analogue of Theorem \ref{wonderland} for discrete trees. 
The reason is that in order to exclude eigenvalues one needs a `free operator' which approximates the tree. 
The natural free operator in the discrete case is the discrete Laplacian which has spectrum in [-2,2]. Since, 
in general, discrete trees might have a significant portion of their spectrum outside [-2,2], the exclusion of 
eigenvalues there does not have the implications of Theorem \ref{simonstolz}.

The rest of this paper is structured as follows. Section 2 describes the reduction of the above theorems to theorems for 
one-dimensional Schr\"odinger operators with point interactions. Section 3 proves a resolvent formula and 
a uniqueness result for such operators. Section 4 completes the proof of the theorems. As noted above, Theorem 
\ref{discrete} is a direct consequence of Theorem 1.1 in \cite{remling-disc}. Thus, no additional discussion will be 
devoted to its proof.

\textbf{Acknowledgments}.
We are grateful to Barry Simon for useful discussions. 
RF appreciates the warm hospitality of Caltech, where part of this work has been done, and acknowledges support through DAAD grant D/06/49117.


\section{Reduction to the one-dimensional case}\label{sec:reduction}

Using the radial symmetry of the tree we shall deduce our main theorems from results about one-dimensional operators. In this section we describe this reduction and state the corresponding theorems in the one-dimensional context.

As in the previous section, let $\Gamma_c$ be a radial metric tree associated with parameters $\{(t_n,b_n)\}_{n=1}^\infty$, which are assumed to satisfy (1). We put $(t_0,b_0):=(0,0)$. For any integer $k\geq 0$, we introduce the self-adjoint operator $A_k^+$ in $L^2(t_k,\infty)$ defined by $(A_k^+ f)(r)=-f''(r)$ for $r\in \cup_{n=k}^\infty (t_n,t_{n+1})$ with domain consisting of all functions 
$f\in H^2\left(\cup_{j=k}^\infty (t_n,t_{n+1})\right)$ satisfying $f(t_k)=0$ and 
\beq \label{boundaries}
f(t_n+)=\sqrt{b_n}f(t_n-), \qquad f'(t_n+)=\frac{1}{\sqrt{b_n}}f'(t_n-)
\eeq
for $n>k$. These operators appear naturally in the direct sum decomposition of the Laplacian \cite{carlson,naim-solom,solomyak}. Indeed, one has

\begin{proposition}[\cite{solomyak}]\label{decomposition}
$-\Delta$ is unitarily equivalent to the direct sum \beq
\label{directsum} -\Delta \cong A^+_0\oplus \sum_{k=1}^\infty \oplus
\left( A_k^+ \otimes I_{\bbC^{b_1\cdots b_{k-1}(b_k-1)}} \right) .\eeq
\end{proposition}

It follows from this proposition that our main results, Theorems \ref{continuous}, \ref{simonstolz} and~\ref{wonderland}, 
will be proved if we can show the corresponding results for any of the operators $A^+_k$, $k\geq 0$. 
Since we consider general sequences $\{(t_n,b_n)\}_{n=1}^\infty$ we may, without loss of generality, restrict our 
attention to $A^+_0$. To simplify notation we denote this operator from now on by $A^+$. 
Moreover, when studying these operators we need no longer assume that the $b_n$'s are integer-valued. All we need is
\begin{enumerate}
\item[(2)] $\inf_{n\geq 1} b_n >1$.
\end{enumerate}
Using boundary conditions \eqref{boundaries} and integrating by parts one easily finds that
\begin{equation}
 \label{eq:apos}
(f,A^+ f) = \int_0^\infty |f'|^2 \,dt
\end{equation}
for $f\in\dom A^+$. In particular, $A^+$ is a non-negative operator.

We now state our results concerning the operators $A^+$ which will imply Theorems~\ref{continuous}, \ref{simonstolz} and~\ref{wonderland}.

\begin{theorem} \label{delta-operator}
Let $\{(t_n,b_n)\}_{n=1}^\infty$ satisfy assumptions (1) and (2) and let $A^+$ be the associated operator. Then, if $\limsup_{n \rightarrow \infty} (t_{n+1}-t_n)=\infty$, the absolutely continuous spectrum of the operator $A^+$ is empty.
\end{theorem}

One easily sees that if $\limsup_{n \rightarrow \infty} (t_{n+1}-t_n)=\infty$, then the spectrum of $A^+$ coincides with the interval $[0,\infty)$. According to Theorem \ref{delta-operator} this spectrum might have a singular continuous and a pure point component. However, if a subsequence of the differences $t_{n+1}-t_n$ grows sufficiently fast, we can rule out the existence of eigenvalues following an argument of Simon and Stolz \cite{simon-stolz} and we obtain 

\begin{theorem}\label{no-ev}
Let $\{(t_n,b_n)\}_{n=1}^\infty$ with $\sup_n b_n<\infty$ and let $A^+$ be the associated operator. Then, if
$ \limsup_{n \rightarrow \infty} n^{-2n} (t_{n+1}-t_n)>0$, the spectrum of $A^+$ coincides with $[0,\infty)$ and is purely singular continuous.
\end{theorem}

Our next result states that singular continuous spectrum is indeed the `generic' situation. 
In order to define what we mean by `generic' we will introduce a natural topology on the sequences 
$\{(t_n,b_n)\}_{n=1}^\infty$ as above. 
It will be convenient to identify such a sequence with a measure $\mu = \sum \beta_n \delta_{t_n}$ on $\R_+$ where $\beta_n =(\sqrt{b_n}+1)/(\sqrt{b_n}-1)$. 
For later use we will consider at once the case of measures on the whole line.

For any $\varepsilon>0$ we denote by $\mathcal{M}_{\rm a}^{\varepsilon}$ the set of all non-negative atomic measures $\mu$ on $\R$ of the form $\mu=\sum_{n\in J} {\beta}_n \delta_{{t}_n}$ where $ \beta_n\in [1,\infty)$ and where $ t_n$ are real numbers satisfying $|{t}_{n}-{t}_{m}| \geq \varepsilon$ for all $n\neq m$. The index set $J$ may be finite, infinite or empty.  Moreover, we denote by $\mathcal{M}_{\rm a}^{\varepsilon,+}$ the subsets consisting of all $\mu\in \mathcal{M}_{\rm a}^{\varepsilon}$ with $\supp\mu\subset [\epsilon,\infty)$, and we put
$$
\mathcal{M}_{\rm a}^{0} := \bigcup_{\varepsilon>0} \mathcal{M}_{\rm a}^{\varepsilon} \,,
\qquad
\mathcal{M}_{\rm a}^{0,+} := \bigcup_{\varepsilon>0} \mathcal{M}_{\rm a}^{\varepsilon,+} \,.
$$
Finally, for $C\geq 2$ let $\mathcal{M}_{\rm a}^{\varepsilon,C,+}$ be the subset consisting of all $\mu\in \mathcal{M}_{\rm a}^{\varepsilon,+}$ with $1+C^{-1} \leq \beta_n \leq C$ for all $n\in J$.

With any measure $\mu= \sum_{n\in J} {\beta}_n \delta_{{t}_n} \in \mathcal{M}_{\rm a}^{0,+}$ we associate an operator $A^+_\mu$ in $L^2(\R_+)$ acting as $A^+_\mu f = -f''$ in $\R_+\setminus\supp\mu$ on functions satisfying $f(0)=0$ and \eqref{boundaries} for all $n\in J$ where $\beta_n = (\sqrt{b_n} +1)/(\sqrt{b_n}-1)$ with $b_n\in (1,\infty]$. (For $b_n=\infty$, \eqref{boundaries} is interpreted as $f(t_n-)=0$ and $f'(t_n+)=0$.) If $J$ is infinite and all $b_n$'s are finite, this is precisely the operator $A^+$ defined above.

The one-dimensional analog of Theorem \ref{wonderland} is

\begin{theorem}\label{wonderland1d}
In the space $\mathcal{M}_{\rm a}^{\varepsilon, C,+}$ with the topology of weak convergence, 
the set of $\mu$'s for which the spectrum of $A_\mu^+$ is purely singular continuous is a dense $G_\delta$ set.
\end{theorem}

\begin{remark*}
 The proof will show that the same is true if we restrict the $b_n$'s to be integers. This is what we need when we deduce Theorem \ref{wonderland}. Moreover, to deduce Theorem \ref{wonderland} we also use that a countable intersection of dense $G_{\delta}$'s is a dense $G_\delta$ by Baire's Category Theorem. 
\end{remark*}


\section{The resolvent and the m-function}

\subsection{A resolvent formula}

In this subsection we derive a convenient expression for the resolvent of the operator $A^+=A^+_\mu$ for $\mu\in
\mathcal{M}_{\rm a}^{0,+}$. We write
\begin{equation}\label{eq:mu}
\mu = \sum_{n\in J} \frac{\sqrt{b_n}+1}{\sqrt{b_n}-1} \ \delta_{t_n}\,,
\qquad
0 < t_1 < t_2 < \ldots \,,
\qquad
1< b_n\leq\infty \,,
\end{equation}
where $J$ is either of the form $\{1,2,\ldots,\#\supp\mu\}$ if there is a finite number of atoms, or $J=\N$ if there are infinitely many atoms. Let $A^+_0 :=-d^2/dt^2$ be the Dirichlet Laplacian in $L^2(\R_+)$ and recall that its resolvent $(A^+_0-z)^{-1}$, $z\in\C\setminus[0,\infty)$, has integral kernel
$$
g_z(t,u) := \frac{i}{2k} \left( e^{ik|t-u|} -e^{ik(t+u)} \right)\,,
\qquad z=k^2,\, \im k>0\,.
$$
Put $\mathfrak H := \ell(J,\C^2)$ and  let $\gamma$ be the trace operator from $L^2(\R_+)$ to $\mathfrak H$ with domain $\dom\gamma=\dom A^+_0$, that is,
$$
(\gamma f)_n := \begin{pmatrix} f(t_n) \\ f'(t_n) \end{pmatrix} \,.
$$
For any $z\in\C\setminus[0,\infty)$ we define an operator $T(z)$ in $\mathfrak H$ by
\begin{align*}
T(z)_{nm} := 
\begin{pmatrix}
    \frac{1}{2ik} \left( e^{ik|t_n-t_m|} -e^{ik(t_n+t_m)} \right) &
    \frac{1}{2} \left( \sigma_{mn} e^{ik|t_n-t_m|} -e^{ik(t_n+t_m)} \right) \\
    \frac{1}{2} \left( \sigma_{nm} e^{ik|t_n-t_m|} -e^{ik(t_n+t_m)} \right) &
    -\frac{ik}{2} \left( e^{ik|t_n-t_m|} + e^{ik(t_n+t_m)} \right)
  \end{pmatrix}\,
\end{align*}
where $\sigma_{mn}:=\sgn(t_m-t_n)$, 
with the convention $\sgn(0)=0$. Finally, we define the multiplication operator $B$ in $\mathfrak H$ by
$$
B_{nm} :=  \delta_{nm} \frac12 \frac{\sqrt {b_n} +1}{\sqrt {b_n}-1}
\begin{pmatrix}
0 & 1 \\ 1 & 0
\end{pmatrix} \,.
$$
The following expression for the resolvent of the operator $A^+=A^+_\mu$ will be very useful for us.

\begin{lemma}\label{res}
For any $z\in\C\setminus[0,\infty)$ one has
\begin{equation}\label{eq:res}
(A^+-z)^{-1} = (A^+_0-z)^{-1}  + \left( \gamma (A^+_0-\overline z)^{-1} \right)^* \left( T(z) + B \right)^{-1}  \gamma (A^+_0-z)^{-1} \,.
\end{equation}
\end{lemma}

\begin{proof}
Obviously, $T(\overline z)=T(z)^*$. It is straightforward to check that
$$
\gamma_\pm \left( \gamma (A^+_0-\overline z)^{-1} \right)^* = -T(z) \pm \frac 12 J\,,
\qquad J_{nm}:=\delta_{nm} 
\begin{pmatrix}
0 & 1 \\ -1 & 0
\end{pmatrix} \,.
$$
Therefore, the resolvent formula implies that
$$
T(z) - T(\zeta) = (\zeta -z) \gamma_{\pm}\left(A^+_0 -\zeta\right)^{-1} \left( \gamma \left(A^+_0 -\overline z\right)^{-1} \right)^* \,
$$
and so
$$
T(z) - T(\zeta) = (\zeta -z) \gamma\left(A^+_0 -\zeta\right)^{-1} \left( \gamma \left(A^+_0 -\overline z\right)^{-1} \right)^* \,.
$$
Hence, by the abstract result of Posilicano \cite{posilicano} there exists a self-adjoint operator $G$, say, with $(G-z)^{-1}$ given by the RHS of \eqref{eq:res}. We need to prove that $G=A^+$. For any function $f\in L^2(\R_+) \cap H^2(\R_+\setminus\{t_n\}_{n\in J})$ we introduce
$$
(\gamma_\pm f)_n := \begin{pmatrix} f(t_n\pm) \\ f'(t_n\pm) \end{pmatrix} \,.
$$
Now assume that $f\in\dom \,G$. From \eqref{eq:res} one sees that $f\in H^2(\R_+\setminus\{t_n\}_{n\in J})$ and  $Gf=-f''$. Let $a_\pm :=\gamma_\pm f$ and $c:=\gamma (A^+_0-z)^{-1} (G-z)f$. (Note that $(A^+_0-z)^{-1} (G-z)f$ and its derivative are continuous.) Applying \eqref{eq:res} to $(G-z)f$ we learn that
\begin{equation}\label{eq:bc}
a_\pm = c +  (-T(z) \pm \tfrac 12 J ) (T(z)+B)^{-1} c = (B \pm \tfrac 12 J ) (T(z)+B)^{-1} c   \,.
\end{equation}
Decomposing $J=\hat J \dot\cup \check J$ where $\hat J := \{n\in J : \, b_n<\infty \}$ and accordingly $\mathfrak H=\hat {\mathfrak H} \oplus \check{\mathfrak H}$ and $a_\pm = \hat a_\pm + \check a_\pm$, we see from \eqref{eq:bc} after eliminating $c$ that
$
\hat a_+ = (B + \frac 12 J )  (B - \frac 12 J )^{-1} \hat a_- \,.
$
(Note that the inverse $(B - \frac 12 J )^{-1}$ is well-defined on $\hat{\mathfrak H}$.) Calculating the product of the two matrices we find jump condition \eqref{boundaries} for $f$. For $n\in\check J$  \eqref{eq:bc} says that the second component of $a_{+,n}$ and the first component of $a_{-,n}$ are zero, which again are the claimed boundary conditions for $f$.
\end{proof}

Our first application of the resolvent formula is to prove that weak convergence of measures implies strong resolvent convergence of the associated operators.

\begin{proposition}\label{strongresconv}
Let $\{\mu^{(j)}\}_{j=1}^\infty \subset \mathcal{M}_{\rm a}^{\varepsilon,+}$ for some $\epsilon>0$ and assume that $\mu^{(j)} \to \mu$ weakly. Then $A^+_{\mu^{(j)}} \to A^+_{\mu}$ in strong resolvent sense.
\end{proposition}

\begin{proof}
Let $f\in L^2(0,\infty)$ with compact support. The assertion will follow if we can prove that for some sufficiently large $\kappa>0$ one has
$$
\left(f, (A^+_{\mu^{(j)}}+\kappa^2)^{-1}f\right) \to \left(f, (A^+_{\mu}+\kappa^2)^{-1}f\right) \,.
$$
(Here we use that weak resolvent convergence is the same as strong resolvent convergence and that the operators $A^+_{\mu^{(j)}}$ and $A^+_{\mu}$ are all non-negative, so that it suffices to verify the convergence at a single point $-\kappa^2$ of the resolvent set.) 
We introduce the operators $T^{(j)}(-\kappa^2)$, $B^{(j)}$ and $\gamma^{(j)}$ in the obvious way and write $a^{(j)} := \gamma^{(j)} (A^+_0 +\kappa^2)^{-1} f$ and $a := \gamma (A^+_0 +\kappa^2)^{-1} f$. In view of the resolvent formula \eqref{eq:res} we need to show that
\begin{equation}\label{eq:convt}
\left( a^{(j)}, \left( T^{(j)}(-\kappa^2) + B^{(j)}\right)^{-1}  a^{(j)} \right)
\to \left( a, \left( T(-\kappa^2) + B \right)^{-1}  a \right) \,.
\end{equation}
We shall assume that $J^{(j)}=J=\N$ for any $j$, the other cases being similar. 
We claim that
\begin{equation}\label{eq:conva}
a^{(j)} \to a
\qquad \text{in} \ \ell^2(\N,\C^2) \,.
\end{equation}
Indeed, the weak convergence implies that $t_n^{(j)} \to t_n$ and hence, since $(A^+_0 +\kappa^2)^{-1} f$ is $C^1$, that $a^{(j)}_n \to a_n$ for each $n$. Moreover, for all $t_n^{(j)}\geq \sup\supp f:=M$ one has $a^{(j)}_n = c e^{-\kappa t_n^{(j)}} (1,- \kappa )$ with a constant $c$ depending on $f$ but not on $j$ or $n$. Since $t_n^{(j)} \geq \epsilon n$, it follows that $|a^{(j)}_n|^2 \leq |c|^2 (1+\kappa^2) e^{- 2\kappa \epsilon n}$ for 
$n\geq M/\varepsilon$ and similarly for $a$. From this one easily deduces \eqref{eq:conva}.

In the proof of Proposition \ref{unique} below we show that for sufficiently large $\kappa$, the operators $\left( T^{(j)}(-\kappa^2) + B^{(j)}\right)^{-1}$ are uniformly bounded in $j$. Moreover, the weak convergence of the measures implies that $T^{(j)}(-\kappa^2) + B^{(j)} \to T(\kappa^2) + B$ strongly. With the help of the resolvent identity one deduces that $\left(T^{(j)}(-\kappa^2) + B^{(j)}\right)^{-1} \to \left( T(\kappa^2) + B\right)^{-1}$ strongly. This, together with \eqref{eq:conva}, implies \eqref{eq:convt}.
\end{proof}

Our second application of \eqref{eq:res} will be to derive an expression for the m-function. Let us recall the definition. 
For later purposes we consider whole-line operators. 
As in Section \ref{sec:reduction} we can associate to each measure $\mu\in \mathcal{M}_{\rm a}^{0}$ on $\R$ a whole line operator $A=A_\mu$ acting as the Laplacian away from $\supp\mu$ on functions satisfying `jump conditions' \eqref{boundaries} for $t_n\in\supp\mu$ and $(\sqrt{b_n}+1)/(\sqrt{b_n}-1)=\mu(\{t_n\})$ (with the same modification as before if $b_n=\infty$). 
This defines a self-adjoint, non-negative operator in $L^2(\R)$. 
Therefore, for any $z\in\C\setminus[0,\infty)$ there exist functions $f_\pm(z;\cdot)$ solving $-f''=zf$ in 
$\R\setminus\supp\mu$, satisfying `jump conditions' \eqref{boundaries} and lying in $L^2$ at $\pm\infty$. 
(For example, choose $f_+(z;t)=\left((A-z)^{-1}g\right)(t)$ where $g$ is supported near $-\infty$ and $t$ is to 
the right of $\supp g$, and continue $f$ to the left.) 
Since $f_\pm$ is defined uniquely only up to a multiplicative constant, it is natural to consider
$$
m_\pm (z;t)=\pm \frac{f'_\pm(z;t)}{f_\pm(z;t)} \,,
$$
the m-functions of $A$. Note that if $\mu(\{0\})=0$ and $t\geq 0$, then $m_\pm(z;t)$ depends only on the restriction of $\mu$ to $\R_+$, and therefore we will also speak of the m-function of $A^+$.

The promised formula is

\begin{corollary}\label{mmatrix}
 Let $\mu\in \mathcal{M}_{\rm a}^{0,+}$ as in \eqref{eq:mu}. Then for all $z=k^2\in\C\setminus[0,\infty)$, $\im k>0$,
\begin{equation}\label{eq:mmatrix}
m_+(k^2;0)=ik + \sum_{n,m} e^{ik(t_n+t_m)}
\begin{pmatrix} 1\\ ik \end{pmatrix} ^T
(T(k^2)+B)^{-1}_{n,m} 
\begin{pmatrix} 1\\ ik \end{pmatrix} \,.
\end{equation}
Here $(T(k^2)+B)^{-1}_{n,m}$ is the $(n,m)$-entry (a $2\times 2$-matrix) of the operator $(T(k^2)+B)^{-1}$.
\end{corollary}

\begin{proof}
Since
$$
m_+(z;0) = \frac{\partial^2}{\partial t \partial u} (A^+-z)^{-1} (t,u) |_{(t,u)=(0,0)} \,,
$$
this follows from \eqref{eq:res}.
\end{proof}



\subsection{Uniqueness}\label{sec:unique}

Our goal in this subsection is to prove that the m-function of $A^+$ uniquely determines the measure $\mu$.

\begin{proposition}\label{unique}
Let $\mu,\tilde\mu \in \mathcal{M}_{\rm a}^{0,+}$ with corresponding operators $A^+=A^+_\mu$, $\tilde A^+=A^+_{\tilde\mu}$. Assume that the corresponding $m$-functions satisfy $m_+(z;t)=\tilde m_+(z;t)$ for some $0\leq t<\min\{\inf\supp\mu,\inf\supp\tilde\mu\}$ and all $z\in\C^+$. Then $\mu=\tilde\mu$.
\end{proposition}

This is an analog of the famous Borg-Marchenko result in the 
Schr\"odinger case. 
It has been generalized to perturbations by measures in \cite{ben-amor}, but the result seems to be new for perturbations by boundary conditions \eqref{boundaries}. Our proof below relies on the expression \eqref{eq:mmatrix} from which we will derive

\begin{lemma}\label{asymp}
Let $0\neq \mu \in \mathcal{M}_{\rm a}^{0,+}$ and put $t_1:=\inf\supp\mu$ and $\mu(\{t_1\}) =: (\sqrt{b_1}+1)/(\sqrt{b_n}-1)$. Then for large, real $\kappa$,
\begin{equation}\label{eq:asymp}
m_+(-\kappa^2;0) + \kappa = - \frac{b_1 -1}{b_1 +1} \ 2 \kappa e^{-2\kappa t_1} (1+o(1)) \,.
\end{equation}
\end{lemma}

Of course, if $b_1=\infty$ then $\frac{b_1 -1}{b_1 +1}=1$. Accepting Lemma \ref{asymp} for the moment, we turn to the

\begin{proof}[Proof of Proposition \ref{unique}]
By translation invariance we may assume that $t=0$. Hence by Lemma \ref{asymp} either both $\mu$ and $\tilde\mu$ are zero, or else they are both not and then $t_1:=\inf\supp\mu = \inf\supp\tilde\mu$ and $\beta_1:= \mu(\{t_1\}) = \tilde\mu(\{t_1\})$. Now choose $t_1<s<\min\{\inf\supp(\mu-\beta_1\delta_{t_1}), \inf\supp(\tilde\mu-\beta_1\delta_{t_1}) \}$. Solving the equation $-f_+''=zf_+$ explicitly on the interval $[0,s]$ we can write $m_+(z;s)$ as a fractional linear function of $m_+(z;0)$ with coefficients depending only on $ s, z, r_1, \beta_1$. Hence $m_+(z;s)=\tilde m_+(z;s)$ for all $z$. Now iterate.
\end{proof}

\begin{proof}[Proof of Lemma \ref{asymp}]
We shall use the expression for $m(-\kappa^2;0)$ from Corollary \ref{mmatrix}. In order to calculate the asymptotics as $\kappa\to\infty$ we decompose
$$
T(-\kappa^2)=T^0(-\kappa^2) + T^R(-\kappa^2) \,,
\qquad
T^0(-\kappa^2)_{nm} : = \delta_{nm}
\begin{pmatrix}
 -\frac{1}{2\kappa} & 0 \\
 0 & \frac{\kappa}{2}
  \end{pmatrix}\,.
$$
One easily estimates for all large $\kappa$
$$
\| T^R(-\kappa^2)_{n,m} \|_{\C^2\to\C^2} \leq
\left\{
\begin{array}{ll}
\const \kappa e^{- 2 \kappa t_n} & \qquad \text{if} \ n=m\,, \\
\const \kappa e^{- \kappa |t_n-t_m|} & \qquad \text{if} \ n\neq m\,.
\end{array}
\right.
$$
Hence by a matrix-valued version of Schur's lemma
$$
\| T^R(-\kappa^2) \|_{\mathfrak H \to\mathfrak H} \leq
\sup_{m} \sum_n \|T^R(-\kappa^2)_{n,m} \|_{\C^2\to\C^2} \leq
\const \kappa \left( e^{- 2 \kappa t_1} + e^{- \kappa \varepsilon} \right) \,,
$$
where we used that $\varepsilon:= \inf_{n\neq m} |t_n-t_m|>0$ and hence $|t_n-t_m|\geq\varepsilon |n-m|$.

On the other hand, the eigenvalues of $T^0(-\kappa)+B$ are easily calculated and one finds that the smallest (in absolute value) eigenvalue is bounded away from zero by a constant times $\kappa^{-1}$ independently of the $b_n$. (To be a bit more explicit, the positive eigenvalue of  $(T^0(-\kappa)+B)_{nn}$ is larger than $\kappa/2$ and the negative eigenvalue is smaller than $-1/2\kappa$.) Hence both $T^0(-\kappa^2) +B$ and $T(-\kappa^2)+B = T^0(-\kappa^2) + B + T^R(-\kappa^2)$ are invertible, and the norms of their inverses are bounded from above by a constant times $\kappa$. We conclude that
\begin{align*}
& \left\| (T(-\kappa^2)+B)^{-1} - (T^0(-\kappa^2)+B)^{-1} \right\| \\
& \qquad = \left\| (T(-\kappa^2)+B)^{-1} \ T^R(-\kappa^2) \ (T^0(-\kappa^2)+B)^{-1} \right\| \\
& \qquad \leq \const \kappa^3 \left( e^{- 2 \kappa t_1} + e^{- \kappa \varepsilon} \right) \,,
\end{align*}
and so by \eqref{eq:mmatrix}
\begin{align*}
m(\kappa,0) & = -\kappa + \sum_{n} e^{-2\kappa t_n }
\begin{pmatrix} 1\\ -\kappa \end{pmatrix} ^*
(T^0(-\kappa^2)+B)^{-1}_{n,n}
\begin{pmatrix} 1\\ -\kappa \end{pmatrix} \\
& \quad \qquad + \mathcal O \left( \kappa^5 e^{-2\kappa t_1} \left( e^{- 2 \kappa t_1} + e^{- \kappa \varepsilon} \right) \right) \\
& = -\kappa - \frac{b_1 -1}{b_1+1}\ 2 \kappa e^{-2\kappa t_1}
+  \mathcal O \left( \kappa^5 e^{-2\kappa t_1} \left( e^{- 2 \kappa t_1} + e^{- \kappa \varepsilon} \right) \right) \,,
\end{align*}
as claimed.
\end{proof}

\begin{remark*}
Looking at the previous proof, we see that the assumption $\inf_{n\neq m} |t_n-t_m|>0$ can be significantly relaxed to 
$\sum_{n:\,n\neq m} e^{-\kappa |t_n-t_m|} = o(\kappa^{-4})$ uniformly in $m$.
\end{remark*}


\section{Proof of Theorems \ref{delta-operator}, \ref{no-ev} and \ref{wonderland1d}}


\subsection{A Remling-type theorem}

Roughly speaking, Remling's Theorem states that for a given
one-dimensional Schr\"odinger operator, $A$, (with a natural
boundedness assumption on the potential) any right-limit of $A$ is
reflectionless on the absolutely continuous spectrum of $A$. Clearly, the
two central notions behind this theorem are that of right-limit and
that of a reflectionless operator. We proceed to define these notions in our 
setting and formulate the version of Remling's Theorem that will be useful for us.

\begin{definition} \label{right-limit}
Let $\mu, \hat{\mu} \in \mathcal{M}_{\rm a}^{0}$. The measure $\hat\mu$ is said to be a \emph{right-limit} of $\mu$ if there exists a strictly increasing sequence, $s_j \rightarrow \infty$, such that for every continuous,
compactly supported function $f$ on $\R$,
\beq \no
\lim_{j \rightarrow \infty}\int_{\bbR}f(t-s_j)d\mu(t)=\int_{\bbR}f(t)d\hat{\mu}(t).
\eeq
\end{definition}

Let $\mu\in \mathcal{M}_{\rm a}^{0,+}$ with $\mu(\{0\})=0$ and let $A$ and $A^+$ be the corresponding whole-line and half-line operators. Recall that we have introduced the m-functions $m_\pm(z;t)$ before Corollary \ref{mmatrix}. Of course, $m_+(z;0)$ depends only on the restriction of $\mu$ to $[0,\infty)$. Since $m_+(z;0)$ is a Herglotz
function of $z\in\C^+$, its boundary values on the real line exist a.e. We denote by $\Sigma_{\rm ac}(A^+)$ the set 
\beq \no 
\Sigma_{\rm ac}(A^+)=\{E \in \bbR \mid 0<\Im m_+(E+i0;0)<\infty\}.
\eeq
$\Sigma_{\rm ac}(A^+)$ is an essential support of the
absolutely continuous spectrum of $A^+$. In particular,
$A^+$ has absolutely continuous spectrum iff $\Sigma_{\rm
ac}(A^+)$ has positive Lebesgue measure (see e.g.\ \cite{rankone}).

\begin{definition} \label{reflectionless}
Fix $\Lambda \subseteq \bbR$ and let $\mu \in \mathcal{M}_{\rm a}^{0}$.
The operator $A_{\mu}$ is called \emph{reflectionless} on $\Lambda$ if for all $t\in \R\setminus\supp\mu$ and for almost 
every $E \in \Lambda$, $m_+(E+i0;t)=-\bar{m}_-(E+i0;t)$.
\end{definition}

We are now in a position to formulate the version of Remling's Theorem
appropriate for our setting.
\begin{theorem}\label{BPR}
Let $A^+=A^+_\mu$ for some $\mu\in \mathcal{M}_{\rm a}^{0,+}$ and let $\hat{\mu}$ be a right-limit of $\mu$. Then 
$A_{\hat{\mu}}$ is reflectionless on $\Sigma_{\rm ac}(A^+)$.
\end{theorem}

\begin{proof}[Proof of Theorem \ref{BPR}]
The proof of this theorem is essentially the same as the proof of
Theorem 4.1 in \cite{remling-cont}, but let us make a few remarks.
As stated in the introduction, Remling's Theorem follows from a
result of Breimesser-Pearson (Theorem 1 in \cite{bp1}) which states
that, on $\Sigma_{\rm ac}(A^+_\mu)$, the value distribution of
$-\frac{v'(E+i0;t)}{v(E+i0;t)}$ and $-\bar{m}_+(E+i0;t)$ are
asymptotically equal (as $t \rightarrow \infty$). Here $v(z;t)$ is
the Dirichlet solution to the formal equation $A^+_\mu v=zv$
(generally, $v \notin L^2$), and $m_+(z;t)$ is the m-function for
$A^+_\mu$. (For the concept of value distribution see
\cite{bp1,bp2}.)

Remling's Theorem follows from this if we can prove that if $\hat{\mu}$ is a right-limit of $\mu$ then
$\lim_{j \rightarrow \infty} m_+(z;s_j)= \hat{m}_+(z;0)$ and
$\lim_{j \rightarrow \infty} -\frac{v'(z;s_j)}{v(z;s_j)}=\hat{m}_-(z;0)$ uniformly on compact subsets of
$\bbC^+$ where $\{s_j\}$ is the sequence from Definition \ref{right-limit} and $\hat{m}_\pm$ are the m-functions
for $A_{\hat{\mu}}$.

To show this, consider $\hat{m}_+(z;0)$ and $A^+_{\hat{\mu}}$,
the right half-line restriction of $A_{\hat{\mu}}$. First, note
that Green's formula (see \cite{cod-lev}): \beq \no \int_{t_1}^{t_2}
(-f''(t))\bar{g}(t)dt-\int_{t_1}^{t_2}f(t)(-\bar{g}''(t))dt
=W\left(f,\bar{g} \right)(t_2)-W\left(f,\bar{g} \right)(t_1) \eeq
holds in our case, by integration by parts along intervals which
contain no atoms of $\hat{\mu}$ and then by summing up the resulting
telescoping sum. Therefore, the Weyl nested disk construction (see,
e.g., \cite{cod-lev}) works in our setting to show that, for any
$\delta>0$ there exists $N>0$ such that, if $\ti{\mu}$ agrees with
$\hat{\mu}$ on $(0,N)$, then $\ti{m}_+(z;0)$, the m-function for
$A^+_{\ti{\mu}}$, lies in a disk of radius no bigger than
$\delta$ which also contains $\hat{m}_+(z;0)$. Explicitly, this disk
is the image of $\bbC^+\cup \bbR$ under the M\"obius transformation
given by the matrix $T_z(N)^{-1}$, where $T_z(N)$ is defined by
$T_z(N) \left( \begin{array}{c} f(0) \\ f'(0)
\end{array} \right)= \left( \begin{array}{c} f(N) \\ f'(N)
\end{array} \right)$ for any formal solution, $f$, of
$A^+_{\hat{\mu}}f=zf$. Using \eqref{boundaries} for the atoms
and the solutions of the free equation along the intervals between
them, it is possible to write this matrix as a product of simple
matrices, and so note that its entries are continuous functions of
the parameters defining the restriction of $\hat{\mu}$ to $(0,N)$.
Thus, the center and radius of the disk are continuous functions of
these parameters. Recalling the definition of right-limit, this
implies the convergence of $m_+(z;s_j)$ to $\hat{m}_+(z;0)$. The
convergence of $-\frac{v'(z;s_j)}{v(z;s_j)}$ to $\hat{m}_-(z;0)$ is
established in a similar way, by considering $A^-_{\hat{\mu}}$, the 
restriction of $A_{\hat{\mu}}$ to $(-\infty,0)$ with a Dirichlet boundary condition.

As for the applicability of Theorem 1 of Breimesser-Pearson
\cite{bp1}, the only thing that actually depends on the particular
properties of the model is Lemma 3 there and its corollary. Once
again, since Green's formula holds in our case it is easy to see
that the proof goes through with no change.
\end{proof}


\subsection{Proof of Theorem \ref{delta-operator}}
Assume, by contradiction, that $\limsup_{n \rightarrow \infty} (t_{n+1}-t_n)=\infty$ and that $\Sigma_{\rm ac}(A^+)$ has positive Lebesgue measure. Note that $A^+$ is non-negative, so $\Sigma_{\rm ac}(A^+) \subset [0,\infty)$.
Let $\mu=\sum_{n=1}^\infty \beta_n \delta_{t_n}$.
By the hypothesis on $t_n$ and a compactness argument, $\mu$ has a right-limit $\hat{\mu}\in\mathcal M_{\mathrm a}^{0}$ such that
$\hat{\mu}(-\infty,0]=0$ while $\hat{\mu}(0,\infty) \neq 0$.

Now, Theorem \ref{BPR} implies that $A_{\hat{\mu}}$ is reflectionless on $\Sigma_{\rm ac}(A^+)$. 
Hence $\hat m_+(E+i0;t)=-\bar{\hat m}_-(E+i0;t)=i\sqrt{E}$ for all $E\in \Sigma_{\rm ac}(A^+)$ and all 
$t\notin\supp\hat\mu$. 
Since a Herglotz function is uniquely determined by its values on a set of positive measure (see, e.g.\ 
\cite[Appendix B]{teschl}), 
one has $\hat m_+(k^2;t)=ik$ for all $k$ with $\im k>0$ and all $t\notin\supp\hat\mu$, and hence by 
Proposition \ref{unique}, $\hat\mu(0,\infty)=0$, a contradiction.
\hfill $\Box$


\subsection{Proof of Theorem \ref{no-ev}}
First note that in view of \eqref{eq:apos}, $0$ is not an eigenvalue of $A^+$. Hence we need to show that any function $f$ satisfying $f(0)=0$ and \eqref{boundaries} and solving
\beq \label{evequation}
-f''(x)=E f(x) \qquad \text{in} \ \R_+\setminus\{t_n\}_{n=1}^\infty
\eeq
for some $E\in(0,\infty)$ is not in $L^2$.

Any solution of \eqref{evequation} satisfies
\beq \no
\left( \begin{array}{c} f(x) \\ f'(x) \end{array} \right)=T(x,y,E)
\left( \begin{array}{c} f(y) \\ f'(y) \end{array} \right)
\eeq
for a certain $2 \times 2$ matrix $T(x,y,E)$ of determinant $1$. It follows from \cite[Thm.~2.1]{simon-stolz} that if
\beq \label{transfer-bound}
\int_0^\infty \frac{dx}{ \parallel T(x,O,E) \parallel^2}=\infty
\eeq
then \eqref{evequation} has no solution in $L^2(0,\infty)$.

If $x, y \in (t_n,t_{n+1})$ for some fixed $n$ then
\beq \no
\left( \begin{array}{c} f(x) \\ f'(x) \end{array} \right)=
\left(\begin{array}{cc} \cos k (x-y) & k^{-1} \sin k(x-y) \\ 
-k \sin k(x-y) & \cos k (x-y) \end{array} \right)
\left( \begin{array}{c} f(y) \\ f'(y) \end{array} \right)
\eeq
where $k=\sqrt{E}>0$. The norm of this matrix is bounded by 
$\max(k,k^{-1})$. Furthermore, the jump condition is taken into account by the 
jump matrix
\beq \no
S_n=\left( \begin{array}{cc} \sqrt{b_n} & 0 \\
0 & \frac{1}{\sqrt{b_n}} \end{array}\right).
\eeq
Thus, if $x \in (t_n,t_{n+1})$,
\beq \no
\parallel T(x,O,E) \parallel \leq \left(\max(k,k^{-1}) \right)^{n+1} \prod_{m=1}^n\sqrt{b_m} 
\leq n^n,
\eeq
for sufficiently large $n$ (and any fixed $E \in (0, \infty)$).
This implies
\beq \no
\int_{t_n}^{t_{n+1}} \frac{dx}{\parallel T(x,O,E) \parallel^2}\geq \frac{t_{n+1}-t_n}{n^{2n}},
\eeq
which, by \eqref{sparseness}, implies the result.
\hfill $\Box$

\subsection{Proof of Theorem \ref{wonderland1d}}
It is easy to see (cf. also \cite{remling-cont}) that $\mathcal{M}^{\varepsilon,C,+}_\mathrm a$ is a complete 
(indeed, compact) metric space where the topology coincides with that of weak convergence of measures. 
According to Proposition \ref{strongresconv} weak convergence of measures implies convergence in the strong resolvent 
sense for the corresponding operators. 
Thus, it follows from \cite[Thm. 1.1]{simon-wonderland} that the 
set of operators in $\mathcal{M}^{\varepsilon,C,+}_\mathrm a$ with no eigenvalues in $[0,\infty)$ is a $G_\delta$ set. 
Moreover, \cite[Thm. 1.2]{simon-wonderland} implies that the set of operators in $\mathcal{M}^{\varepsilon,C,+}_\mathrm a$ 
with no absolutely continuous spectrum in $(0,\infty)$ (and so also in $[0,\infty)$) is a $G_\delta$ set.

To complete the proof it suffices to note that the set of measures in 
$\mathcal{M}^{\varepsilon,C,+}_\mathrm a$ with $\{t_n\}$ satisfying \eqref{sparseness} is a dense set 
(since, for any given measure one may take a measure coinciding with it on a fixed bounded set, 
but satisfying condition \eqref{sparseness}), and for such measures the spectrum of the corresponding operator is purely 
singular continuous by Theorem \ref{no-ev}.
\hfill $\Box$


\bibliographystyle{amsalpha}

\end{document}